\documentclass[12pt]{article}
\usepackage{graphics}
\usepackage{graphicx}
\usepackage{amsmath}
\usepackage{amsfonts}
\usepackage{amssymb}
\usepackage{amsthm}
\usepackage{latexsym}
\usepackage{hyperref}
   \usepackage{epstopdf}
\usepackage{amsmath}
\usepackage{amssymb}
\usepackage{amsfonts}
\usepackage{amsbsy}
\usepackage{bbm}
\usepackage[]{algorithm2e}

\usepackage{fullpage}

\usepackage{tikz}
\tikzset{every node/.style={circle,draw,inner sep=2pt}}
\tikzset{every label/.style={rectangle,draw=none}}

\usepackage{makecell}

\usepackage{geometry}
\usepackage{graphicx}

\graphicspath{{./figs/}}

\newcommand{\A}{\mathbf{A}}
\newcommand{\vx}{\mathbf{x}}
\newcommand{\vzero}{\mathbf{0}}

\newtheorem{theorem}{Theorem}
\newtheorem{prop}[theorem]{Proposition}

\theoremstyle{definition}
\newtheorem{problem}[theorem]{Problem}

\newtheorem{mydef}[theorem]{Definition}

\newcommand{\graphbox}[5]%[-5, 5, -5, 5, 0.33]
{
\begin{tikzpicture}
     [>=latex,scale=#5]
     
     \draw [->,very thick] (#1, 0) -- (#2, 0) node[right] {$x$};
     \draw [->,very thick] (0, #3) -- (0, #4) node[above] {$y$};
     
     \draw[step=1cm,thick,dotted] (#1,#3) grid (#2,#4);
   \end{tikzpicture}
   }

\begin{document}

%   \begin{frontmatter}
 
\title{On the error of {\it a priori} sampling: zero forcing sets and propagation time}
 \author{Franklin H. J. Kenter\footnote{United States Naval Academy, Department of Mathematics; Annapolis, MD, USA (kenter@usna.edu)} \and Jephian C.-H. Lin \footnote{University of Victoria, Department of Mathematics and Statistics; Victoria, BC, Canada (chinhunglin@uvic.ca)}}

\maketitle

% \cortext[cor1]{Corresponding author}
% \address[fhj]{United States Naval Academy, Department of Mathematics; Annapolis, MD, USA}
%  \address[jchl]{University of Victoria, Department of Mathematics and Statistics; Victoria, BC, Canada}
 
\begin{abstract}
% We present new randomized spectral algorithms for several combinatorial problems 
 
 %including the densest subgraph problem, discrepancy, and graph coloring. 
%These algorithms produce spectral bounds which tightly control the parameters to within a linear factor of the corresponding eigenvalue. Most interestingly, these bounds depend upon the 1- and $\infty$-norms of the corresponding eigenvector.
Zero forcing is an iterative process on a graph used to bound the maximum nullity. The process begins with select vertices as colored, and the remaining vertices can become colored under a specific color change rule. The goal is to find a minimum set of vertices such that after iteratively applying the rule, all of the vertices become colored (i.e., a minimum zero forcing set). 
Of particular interest is the propagation time of a chosen set which is the number of steps the rule must be applied in order to color all the vertices of a graph. 

We give a purely linear algebraic interpretation of zero forcing: Find a set of vertices $S$ such that for any weighted adjacency matrix $\mathbf{A}$, whenever $\mathbf{Ax} = \mathbf{0}$, the entirety of of $\mathbf{x}$ can be recovered using only $\mathbf{x}_S$, the entries corresponding to $S$. The key here is that $S$ must be chosen before $\mathbf{A}$. In this light, we are able to give a linear algebraic interpretation of the propagation time: Any error in $\mathbf{x}_S$ effects the error of $\mathbf{x}$ exponentially in the propagation time. This error can be quantitatively measured using newly defined zero forcing-related parameters, the error polynomial vector and the variance polynomial vector. In this sense, the quality of two zero forcing sets can objectively be compared even if the sets are the same size and their propagation time is the same. Examples and constructions are given.
\end{abstract}

%\begin{keyword}
% minimum rank problem; zero forcing; propagation time; error polynomial vector; variance polynomial vector
%  
%AMS Classification:  05C50, 05C57, 15A18
%\end{keyword}
% 
%\end{frontmatter}

\section{Introduction}

Zero forcing is a one-player game introduced to bound the maximum nullity of a graph \cite{zf_aim}. Given a graph, $G = (V, E)$, the player selects a set of vertices $S$ to color blue, then iteratively applies the color change rule: Identify each colored vertex $u$ that has only one uncolored neighbor $v$, then color all such $v$; in which case, we say that $u$ {\it forces} $v$. The goal of the player is to find the smallest set initial set $S$ such that eventually all the vertices become colored, such a set is called a {\it zero forcing set}. The minimum size of all zero forcing sets is the {\it zero forcing number} of $G$, denoted $Z(G)$. The {\it propagation time} of a graph $G$ with zero forcing set $S$, denoted ${pt}(G,S)$ is the number of iterations of the simultaneous application of the color change rule needed in order force the entire graph to be colored; for emphasis, multiple vertices may force at the same time. 

For a graph $G = (V, E)$, let $\mathcal{S}_F(G)$ denote the set of all $|V| \times |V|$ matrices, $\mathbf{A}$, over the field $F$ with $\mathbf{A}_{ij} = 0$ whenever $i \ne j$ and $\{i,j\} \not\in E$, $\mathbf{A}_{ij} \ne 0$ when $\{ij\} \in E$  and the diagonal entries may take any value. The maximum nullity of $G$ over $F$, $M_F(G)$, is the maximum possible dimension of the null space of $\mathbf{A}$ over all $\mathbf{A} \in \mathcal{S}_F(G)$. The following relates the maximum nullity and the zero forcing number.

\begin{prop}[AIM Group (2008) \cite{zf_aim}]
For any graph $G$ and any field $F$, 
\[ M_F(G) \le Z(G). \]
\end{prop}

Recently, the study of the propagation time of zero forcing and its variations has grown \cite{butlerthrot, semithrot, powertimenote, hogbenprop, semiprop}. However, these studies appear to be motivated on their own interest without any connection to linear algebra. One interesting aspect of propagation time is the phenomenon of ``throttling'' where the quantity $|S| + pt(G,S)$ is minimized using non-minimal zero forcing sets. That is, near-minimum zero forcing sets can have substantially faster propagation times compared to minimum zero forcing sets \cite{butlerthrot, semithrot}.

The motivation for this study is to establish a concrete linear-algebraic interpretation of propagation time. Our main result, interprets the propagation time as an exponential growth of error in a certain setting. As a result, one can more carefully quantify the trade-offs between smaller zero forcing sets with slow propagation times and larger zero forcing sets with fast propagation times. In particular, while we do not propose a specific metric, we provide the tools to choose a more appropriate metric besides the seemingly arbitrary $|S| + pt(G,S)$ mentioned above. 

Before describing our result, we need to take a different perspective of zero forcing using a linear algebraic approach. Consider the following problem:

\begin{problem}[Posing zero forcing linear algebraically] \label{linzf1}
Given a graph $G$ and any infinite field $F$, find a minimum set of indices $Q$ such that \emph{for any} $\mathbf{A} \in \mathcal{S}_F(G)$, \emph{any} vector $\mathbf{x}$ with $\mathbf{A x}= \mathbf{0}$ can be uniquely determined by knowing $\mathbf{x}_Q$, the entries of $\mathbf{x}$ corresponding to $Q$, and $\mathbf{A}$; for emphasis, the indices $Q$ must be chosen {\it before} knowing $\mathbf{A}$. 
\end{problem}
A more applied view of this problem is the following: Given a schematic of a system (i.e., the sparsity pattern), where should one place the sensors (i.e., $Q$) so that once the details of the system are known later, the entirety of the system can be determined using only the measurements from the sensors? Indeed, the concept of having to choose where to measure before knowing the exact details of the system appears in the ever-growing applications of zero forcing including monitoring power grids \cite{pmu} and measuring quantum systems \cite{zf.quantum2}, and analyzing the controllability of consensus dynamics \cite{topdynamics}.

Of course, using this view, if $S$ is a zero forcing set of $G$, then by taking $Q$ to be $S$, one can use the entries of $\mathbf{x}_Q$ and the zero forcing color change rule allows one to ``backsolve'' uniquely for the entirety of $\mathbf{x}$. However, the converse is not as obvious; in Section \ref{zproof} we will prove the following:

\begin{prop}[Zero forcing reimagined linear algebraically]\label{linzf2} Given a graph $G$ and a field $F$ (except $\mathbb{F}_2$), the following are equivalent:
\begin{enumerate}
\item $S$ is a zero forcing set of $G$.
\item For any matrix $\A\in\mathcal{S}_F(G)$, the columns corresponding to $V (G)\setminus S$ are linearly independent.
\item For any matrix $\A\in\mathcal{S}_F(G)$, whenever $\A\vx = \vzero$, the vector $\vx_S$ is always sufficient to uniquely determine the entirety of $\vx$.
\end{enumerate}
\end{prop}
%The key idea to keep in mind is that the solution to the graph theory problem of zero forcing is the same as the linear algebraic problem above. Notably,

The viewpoint that the zero forcing can be viewed as ``placing sensors'' before one know the details of the system is essential to our interpretation of propagation time. Consider the following variation of Problem \ref{linzf1} now posed with error in measurement:

\begin{problem}[Zero forcing reimagined linear algebraically with error] \label{linzf3}
Let $G$ be a graph and $S$ a zero forcing set of $G$. Suppose $\mathbf{A x}= \mathbf{0}$ with $\A \in \mathcal{S}_{\mathbb{R}}(G)$, and a vector $\mathbf{x}'_S$ is given with absolute error $\|\mathbf{x}'_S - \mathbf{x}_S\|_\infty < \varepsilon$. Using $\mathbf{x}'_S$, can you compute a vector $\hat{\mathbf{x}}$ such that the absolute error  $\|\mathbf{x} -\hat{\mathbf{x}}\|_\infty$ is small and $\hat{\mathbf{x}}_S = \mathbf{x}'_S$.
\end{problem}

%This provides the key distinction between the maximum nullity and zero forcing. For the maximum nullity, one knows the matrix before computing a basis for the kernel whereas for zero forcing, one needs to be able to determine ``a cover'' for the basis before knowing any details about the matrix. 
%
%\begin{prop}
%For Problem \ref{linzf}, when $\mathcal{F} = \mathcal{S}(G)$, any such $S$ is a zero forcing set of $G$, and vice versa.
%\end{prop}
%
%It will be helpful to understand the idea behind proof. For any matrix $\mathbf{A}$, suppose one wishes to determine a unique solution to $\mathbf{A x} = \mathbf{0}$ given only select entries of $\mathbf{x}$. Ordinarily, this can be achieved by finding a reduced form of $\mathbf{A}$ then using it to find a basis of the kernel in terms of the free variables. If the set of given entries includes the free variables, then a unique solution is given, and if not, a unique solution cannot be found. On the other hand, the zero forcing guarantees that if certain entries of $\mathbf{x}$ are known, then the remainder can be determined by backsolving with $\mathbf{A}$. Hence, we have the following:
%
%\begin{prop}
%Let $G$ be a graph. $S$ is zero forcing set of $G$ if and only if for any field $F$ and any $\mathbf{A} \in \mathcal{S}_F(G)$, $S$ is a subset of the vertices corresponding to the columns of the non-pivot positions of the reduced echelon form of $\mathbf{A}$. \hfill \qed
%\end{prop}

Our main result shows that if one chooses the vertices to monitor (i.e. a zero forcing set), then any error in the measurement at those vertices propagates exponentially in the propagation time based on two other parameters, $\Delta$, the maximum degree of $G$, and the multiplicative row-support spread, $\kappa'(\mathbf{A})$, defined as follows.

\begin{mydef}[Multiplicative row-support spread]
For an $m \times n$ matrix $\mathbf{A}$, define
\[ \kappa'(\mathbf{A}) := \max_{i,j,k} \left| \frac{\mathbf{A}_{ij}}{\mathbf{A}_{ik}} \right| \]
where the maximum is taken over entires where both $\mathbf{A}_{ij},\mathbf{A}_{ik}$ are nonzero. As a convention for the zero matrix, $\kappa'(\mathbf{0}_{m \times n})=1$.
\end{mydef}

%To first understand the significance of our result, it is important to take a slightly different viewpoint of zero forcing. Indeed, the maximum nullity of a graph is the maximum nullity over all matrix representations of the graph. Put another way, it is is the maximum number of pivot positions of the reduced echelon form of the matrix formed by the basis for the null space over all representations. On the other hand, a zero forcing set must contain a set of pivot positions for all matrices of a graph. In more applied terms, if the nullity of $\mathbf{A}$ is $k$, then a unique solution of $\mathbf{A x} = \mathbf{0}$ can be determined by only knowing a $k$ specific entries of the vector $x$. However, if a graph has a zero forcing set of size $k$, then for {\it any} matrix $\mathbf{A} \in \mathcal{S}(G)$, the equation

\begin{theorem}\label{laprop}
Let $G$ be a graph with maximum degree $\Delta$. Let $F$ be $\mathbb{R}$ (or $\mathbb C$). Fix a zero forcing set $S$ with propagation time $\tau$. \textbf{For any} $\mathbf{A} \in \mathcal{S}_F(G)$ if $\mathbf{A} \mathbf{x} =  \mathbf{0}$ and $\| \mathbf{x} -\mathbf{x}'\|_\infty < \varepsilon$, then only using $\mathbf{x}_S'$, the entries of $\mathbf{x}'$ corresponding to $S$, and $\mathbf{A}$,
a vector $\hat{\mathbf{x}}$ can be computed with  
$\| \mathbf{x} -\hat{\mathbf{x}}\|_\infty \le [\kappa'(\mathbf{A}) \Delta]^\tau \varepsilon$ and $\hat{\mathbf{x}}_S = \mathbf{x}'_S$.

\end{theorem}

It will be helpful to understand what this theorem actually says. Here, and throughout the paper, $\mathbf{x}$ represents the true solution to the equation $\mathbf{Ax} = \mathbf{0}$, $\mathbf{x'}$ represents the measured values $\mathbf{x}$ (with error), and $\hat{\mathbf{x}}$ represents a complete vector you can compute using only selected entries of $\mathbf{x}'$. The theorem says that if you select a set of entries of $\mathbf{x'}$ that correspond to a zero forcing set, not only can you compute $\hat{\mathbf{x}}$, but the error between $\hat{\mathbf{x}}$ and $\mathbf{x}$ is limited as quantified in the theorem. Further,  sampling the same entries of $\mathbf{x'}$ will work regardless of the $\mathbf{A}$ given.

%\begin{abstract}
%\end{abstract}

% COMMENTED OUT LEMMA!
%\begin{lemma}
%Suppose $S$ is a zero forcing set of a graph $G$ with propagation time $T$. Let $S = S_0 \subset S_i \subset \ldots S_T = V(G)$ be the chain of subsets where $S_i$ is the set of colored vertices at time step $i$. Then, for any $k$, $S_k$ is a zero forcing set of $G$, and further, for $j < k$, $S_j$ is a zero forcing set for the graph induced on $S_k$. \hfill \qed
%\end{lemma}

\section{Zero Forcing Imagined Linear Algebraically} \label{zproof}

\subsection{Proof of Proposition \ref{linzf2}}

\begin{proof}
Using the definition of zero forcing, (1) implies (2).

Additionally, (2) and (3) are equivalent by the subspace decomposition: $\mathbb{R}^n = \mathcal{W}_{ker} \bigoplus \mathcal{W}_{coimg}$.
Therefore, it suffices to prove ``not (1) implies not (2)''.

Suppose $S \subset V$ is not a zero forcing set.  Then, we can keep applying the color change rule until there is no more vertices can be colored.  Let $X$ be the set of blue vertices at this point, which is also called the derived set.  Since $S$ is not a zero forcing set, $Y:=V(G)\setminus X$ is not empty. 

For the remainder, given sets of indices $S$ and $T$, we will denote the matrix $\mathbf{A}$ restricted to the columns indicated by $S$ as  $\mathbf{A}_{:,S}$, and similarly, we will let $\mathbf{A}_{S,T}$ denote the matrix $\A$ restricted to the rows corresponding to $S$ and columns corresponding to  $T$. 

We will construct a matrix $\A\in\mathcal{S}(G)$ such that $\mathbf{A}_{:,Y}$ has zero row sum on each row, so the columns of $\mathbf{A}_{:,Y}$ are dependent.  First, we may set $\mathbf{A}_{Y,Y}$ as the Laplacian matrix of the graph induced on $Y$, so each row of $\mathbf{A}_{:,Y}$ that corresponds to $Y$ has zero row sum. 

Since there is no more forcing to do, every vertex in $X$ (the set of blue vertices)  has either no or at least two neighbors in $Y$ (the set of white vertices).  This means each row of $\mathbf{A}_{X,Y}$ either has all entry zero or at least two nonzero entries.  If the row has all entry zero, then it has zero row sum already.  If the row has $k$ nonzero entries, say ${x}_1,\ldots, {x}_k$, then we solve the equation
\[x_1+x_2+\cdots x_k=0, x_i\neq 0~(i=1,\ldots, k)\]
by the following process.  First let $x_1=\cdots=x_{k-2}=1$.  Next pick a nonzero value for $x_{k-1}$ such that 
\[x_1+\cdots +x_{k-1}\neq 0.\]
(This can be done as long as the field has at least two nonzero elements, or equivalently, $F \not \cong \mathbb{F}_2$.)
Finally, pick 
\[ x_k=-(x_1+\cdots +x_{k-1}),\]
which is nonzero.  Doing this process to all the rows that correspond to $X$, then we found a matrix $\mathbf{A}_{:,Y}$ with zero row sums.  And it is not hard to extend $\mathbf{A}_{:,Y}$ to a matrix $\A\in\mathcal{S}(G)$.
\end{proof}

\subsection{A Counterexample for $F=\mathbb{F}_2$}

We present a counterexample to Theorem \ref{linzf2} when $F$ is the field of two elements. Consider $G$ as obtained by removing two disjoint edges from $K_6$, say $\{1, 5\}$ and $\{2, 6\}$.  Thus, the matrix in $\mathcal{S}_{{\mathbb{F}}_2}(G)$ necessarily looks like
\[\A=\begin{bmatrix}
x_1 & 1 & 1 & 1 & 0 & 1 \\
1 & x_2 & 1 & 1 & 1 & 0 \\
1 & 1 & x_3 & 1 & 1 & 1 \\
1 & 1 & 1 & x_4 & 1 & 1 \\
0 & 1 & 1 & 1 & x_5 & 1 \\
1 & 0 & 1 & 1 & 1 & x_6 \\
\end{bmatrix},\]
where $x_1,\ldots ,x_6\in {\mathbb{F}}_2$.  For $\A$, the first three columns are always linearly independent since 
\[A[\{4,5,6\},\{1,2,3\}]=\begin{bmatrix}
1&1&1\\
0&1&1\\
1&0&1
\end{bmatrix}\]
is a nonsingular matrix over $\mathbb{F}_2$. However, $\{4,5,6\}=V(G)\setminus\{1,2,3\}$ is not a zero forcing set.  Indeed, $Z(G)=4$ but not $3$, by brute force or the Minimum Rank Software Library in \textit{Sage} \cite{mr_software} or other software (see sources within \cite{compzf}).

\section{Proof of the Main Result}

\begin{proof}[Proof of Theorem \ref{laprop}]

We proceed by induction on the propagation time, $\tau$. We will let $S_{i}$ denote the set of vertices colored after $i$ iterations of the color change rule with $S = S_0$.

Base case ($\tau=0$): If $\tau = 0$, then $S  = S_0 =  V(G)$, and so one can set $\mathbf{x}' = \hat{\mathbf{x}}$ and the theorem follows.

Induction step: Suppose the theorem holds for $\tau \le t-1 $, we will show it holds for $\tau = t$.

Since $S$ is a zero-forcing set for $G$, it is also a zero forcing set for the graph induced on $S_{\tau-1}$. Hence, by using the induction hypothesis, we can define $\hat{\mathbf{x}}$ on the vector components corresponding to the vertices $i \in S_{\tau-1}$ so that $| \mathbf{x}_i - \hat{\mathbf{x}}_i | \le [\kappa'(\mathbf{A}) \Delta]^{\tau-1} \varepsilon$.

Since $S_{\tau-1}$ is a zero forcing set of $G$ with propagation time of 1, for every $k \in S_\tau - S_{\tau-1}$, there is some $i \in S_{\tau-1}$ that forces $k$. Hence, all the neighbors of $i$, except for $k$ are colored, and in particular, $\hat{\mathbf{x}}_j$ is well-defined (using the induction hypothesis) for all neighbors of $i$,  $j \ne k$. Therefore, we can define
\begin{equation}
 \hat{\mathbf{x}}_{k} := \frac{-1}{\mathbf{A}_{ik}} \sum_{j \in N[i] \setminus \{k\}}\mathbf{A}_{ij} \hat{\mathbf{x}}_j  \label{xhatk}.
\end{equation}

It remains to show that for all such vertices  $k \in S_\tau - S_{\tau-1}$, $| \mathbf{x}_k - \hat{\mathbf{x}}_k | \le [\kappa'(\mathbf{A}) \Delta]^{\tau} \varepsilon.$

For $k \in S_\tau$ that is forced by $i$, using the $i$-th row of $\mathbf{Ax} = \mathbf{0}$, we have:

\begin{eqnarray}
  0 &=& \sum_{{j \in N[i]}}\mathbf{A}_{ij} \mathbf{x}_j \notag \\
  -\mathbf{A}_{ik} \mathbf{x}_k &=& \sum_{{j \in N[i] \setminus \{k\}}}\mathbf{A}_{ij} \mathbf{x}_j \notag \\
 \mathbf{x}_k &=& \frac{-1}{\mathbf{A}_{ik}} \sum_{j \in N[i] \setminus \{k\}} \mathbf{A}_{ij} \mathbf{x}_j  \label{xk}
 \end{eqnarray}
 
By combining (\ref{xhatk}) and (\ref{xk}), we have
 
\begin{eqnarray}
| {\mathbf{x}}_{k} - \hat{\mathbf{x}}_k  | &=&  \left| \frac{-1}{\mathbf{A}_{ik}} \sum_{{j \in N[i] \setminus \{k\}}} \mathbf{A}_{ij} \mathbf{x}_j - \frac{-1}{\mathbf{A}_{ik}}  \sum_{{j \in N[i] \setminus \{k\}}}\mathbf{A}_{ij} \hat{\mathbf{x}}_j  \right| \label{r} \\
 &\le&  \sum_{j \in N[i] \setminus \{k\}} \left|\frac{\mathbf{A}_{ij}}{\mathbf{A}_{ik}}\right| |\mathbf{x}_j-\hat{\mathbf{x}}_j| \label{q} \\
 &\le& \sum_{j \in N[i] \setminus \{k\}} \kappa'(\mathbf{A}) [\kappa'(\mathbf{A}) \Delta]^{\tau-1} \varepsilon  \notag \\
 &\le&  [\kappa'(\mathbf{A}) \Delta]^{\tau} \varepsilon. \notag
\end{eqnarray}
 
Above, the fourth line follows from the definition of $\kappa'(\mathbf{A})$ and the induction hypothesis, and the final line follows from the fact the sum has at most $\Delta$ terms. \end{proof}

%\begin{question}
%For a graph $G$ on $n$ vertices, a set of vectors $\{ \mathbf{V}_1, \ldots \}$ with $\mathbf{V}_i \in \mathbb{R}^n$ is called \emph{forcing} if for every vertex
%\end{question}

\section{Notes on Theorem \ref{laprop}}

\subsection{$\kappa'(G)$ is necessary in the bound.}

The parameter $\kappa'(G)$ is seeming unaesthetic in the context of the minimum rank problem; however, we will show its use is unavoidable, as the error depends on the matrix chosen in $\mathcal{S}(G)$. For example, consider $G = K_n$, the complete graph on $n$ vertices. Let $\mathbf{A}_{1,k} =\mathbf{A}_{k,1} = \delta > 0$ for all $k \ne 1$, $\mathbf{A}_{1,1} = \delta^2$ and $\mathbf{A}_{ij} = 1$ otherwise. If we take $S = V \setminus \{1\}$ to be a zero forcing set, then regardless of which vertex forces vertex $1$, we have

\[ \hat{\mathbf{x}}_1 = \frac{-1}{\delta}  \sum_{k \ne 1} \mathbf{x}'_k \]

Therefore, the error can be given by:

\[ \left| \mathbf{x}_1 - \hat{\mathbf{x}}_1 \right| =  \frac{1}{\delta} \left | \sum_{k \ne 1} \mathbf{x}_k - \mathbf{x}'_k \right |. \]

If the error for each $\mathbf{x}_k$ is the the maximum allowed, say  $\mathbf{x}_k - \mathbf{x}'_k = \varepsilon$, then the total error for $\hat{\mathbf{x}}_1$ is ${ \varepsilon \Delta} / {\delta}$. Since $\delta$ can be arbitrarily small, this demonstrates that the error can be arbitrarily large unless the matrix is otherwise constrained.

\subsection{The exponential bound is necessary}

Consider the path on $n$ vertices labelled $1, \ldots, n$ with $n$ odd. Take $\mathbf{A}_{ij} = 2^{\min(i,j)}$ whenever $i$ and $j$ differ by exactly 1 and $\mathbf{A}_{ij} = 0$ otherwise. Indeed, the nullity of $\A$ is 1 (maximum possible) and a nonzero null vector of $\mathbf{A}$ is given by $\mathbf{v}_i = (-2)^{(n-i)/2}$ for $i$ odd and $\mathbf{v}_i = 0$ for $i$ even. If the zero forcing set is taken to be $S = \{n\}$ (i.e., the end of a path), then any error in $x_n$ will propagate by a factor of 2 at every other step.

\section{Beyond Propagation Time: Polynomial Vectors}

The estimate in Theorem \ref{laprop2} can be quite wasteful. Indeed, it is not necessarily the case that whenever vertex $i$ forces vertex $k$ at time step $t$ that all of the neighbors of $i$ have error $[\Delta \kappa'(G)]^t$. To get a more precise measure of the error produced by a zero forcing set, we introduce the concept of the {\it error polynomial vector}.

We need to define the concept of a forcing chain: Given a zero forcing set of $G$, a {\it forcing chain} is set of directed paths that covers $V(G)$ and $i \to j$ among these paths only if $i$ forces $j$. Note that the starts of all the paths correspond directly with the zero forcing set. In essence, the forcing chain details the exact strategy as to which vertices force which other vertices, if for instance, there are two or more vertices can force a particular vertex.

\begin{mydef}[Error polynomial vector of a forcing chain] 
Let $S$ be a zero forcing set of $G$ with forcing chain $S'$.
We define the {\it error polynomial vector} of $S'$, $\mathbf{q}^{S'}(t)$, as a vector of polynomials, recursively, as follows:

\[ \mathbf{q}^{S'}_k (t) = \left\{ \begin{array}{cc} \displaystyle 1 &\text{ if } k \in S \\
\displaystyle t \sum_{j \in N[i] \setminus \{k\}} \mathbf{q}^{S'}_j(t)   &\text{ if } i \text { forces } k  \end{array} \right. \]
\end{mydef}

Note that $\mathbf{q}^{S'}(t)$ is well-defined, as each vertex $i$ is either in $S$ or is forced by a unique vertex $k$ as determined by the forcing chain $S'$.

\begin{theorem}\label{laprop2}
Let $G$ be a graph, and let $F$ be $\mathbb{R}$ (or $\mathbb C$). Fix a zero forcing set $S$ with forcing chain $S'$ and error polynomial vector $\mathbf{q}^{S'}(t)$. Then, \textbf{for any} $\mathbf{A}\in \mathcal{S}_F(G)$  with multiplicative row-support spread $\kappa' := \kappa'(\mathbf{A})$, if $\mathbf{A} \mathbf{x} =  \mathbf{0}$ and $\| \mathbf{x} -\mathbf{x}'\|_\infty < \varepsilon$, then only using $\mathbf{x}_S'$, the entries of $\mathbf{x}'$ corresponding to $S$, and $\mathbf{A}$,
a vector $\hat{\mathbf{x}}$ can be computed with  
$| \mathbf{x}_i -\hat{\mathbf{x}}_i | \le\mathbf{q}^{S'}_i( \kappa') ~ \varepsilon$ for all $i$ and $\hat{\mathbf{x}}_S = \mathbf{x}'_S$.
\end{theorem}

\begin{proof}
The proof is very similar to that of Theorem \ref{laprop}. After line (\ref{q}), we have

\begin{eqnarray}
 &\le&  \sum_{j \in N[i] \setminus \{k\}} \left| \frac{\mathbf{A}_{ij}}{\mathbf{A}_{ik}} \right| |\mathbf{x}_j-\hat{\mathbf{x}}_j| \notag \\
 &\le& \sum_{j \in N[i] \setminus \{k\}} \kappa' \mathbf{q}^{S'}_j(\kappa') ~ \varepsilon  \notag \\
 &\le& \mathbf{q}^{S'}_k(\kappa') ~ \varepsilon \notag.
\end{eqnarray}
Above, the second line follows from the corresponding induction hypothesis and final line follows from the definition of $\mathbf{q}^{S'}(t)$  .
\end{proof}

It should be emphasized that for a zero forcing set $S$, the choice of $S'$ indeed matters. For instance, if two vertices can force the vertex $k$, one may have a different value of $\mathbf{q}^{S'}(t) (t)$ than the other. To compare two entries of $\mathbf{q}^{S'}(t)$ (i.e., polynomials in $t$), we say $p(t) \preceq r(t)$, if $p(t) \le r(t)$ for sufficiently large $t$ (i.e., $r$ has greater coefficients for the highest degree(s) until the first non-equal coefficient). This choice may seem arbitrary, and indeed, as in Theorem \ref{laprop2}, the total error is not solely determined by the highest power. However, since  $\kappa'(\mathbf{A})$ is the input for the polynomial vector, necessarily $\kappa'(\mathbf{A}) \ge 1$, and $\kappa'(\mathbf{A})$ can be arbitrarily large, it makes sense to use this ordering based on the higher coefficients.

\begin{mydef}[Error polynomial vector of a zero forcing set]
Given a graph $G$ and a zero forcing set $S$, define the error polynomial vector of $S$
\[ \mathbf{q}^{S}_i (t) = \min_{S',\preceq} \mathbf{q}^{S'}_i(t) \]
where the minimum is taken over all forcing chains $S'$ using the zero forcing set $S$, under the order $\preceq$.
\end{mydef}

For emphasis, we define the error polynomial for a zero forcing set as the entry-wise minimum over all forcing chains. However, we now show that there is a chain that achieves the minimum for all vertices.

\begin{prop} \label{sameq} Given a graph $G$ and a zero forcing set $S$, there is a forcing chain $S'$ of $S$ such that $\mathbf{q}^{S}(t) = \mathbf{q}^{S'}(t)$.
\end{prop}

\begin{proof} We construct the forcing chain $S'$ as follows. Whenever a vertex $k$ can be forced, choose the vertex $i$, among all possible vertices that can force $k$, such that \[ \sum_{j \in N[i]\setminus\{k\}} \mathbf{q}^{S'}_j (t) \] is minimized under $\preceq$. Observe that the degree of $\mathbf{q}^{S'}_k(t)$ is the time index at which $k$ becomes colored (or 0 if $k \in S$). Hence, waiting for a subsequent vertex to force $k$ will result is a greater polynomial under the order $\preceq$. Therefore, by taking the minimum choice at the minimum possible time step, the entry of $\mathbf{q}_i^{S'}(t)$ is guaranteed to be minimum.
\end{proof}

In short, Proposition \ref{sameq} says that the ``best'' forcing chain is determined by the initial zero forcing set $S$ as determined in a greedy manner.

Theorems \ref{laprop} and \ref{laprop2} are interested in the worse possible error. In many applications, the input error is not absolute, rather, it is random. As we will discuss going forward, the variance of this error may be different in random setting than in the absolute setting. 

With regard to the probability we will use below, we will let $\mathbb{E}[X]$ denote the expectation of the random variable $X$ and ${\rm Var}[X] := \mathbb{E} (X-\mathbb{E}[X])^2$ to be the variance of $X$.

\begin{mydef}[Multivariable error polynomial vector of a forcing chain] 
Let $S$ be a zero forcing set of $G$ with forcing chain $S'$.
We define the {\it multivariable error polynomial vector} of $S'$, $\mathbf{q}^{S'}$ as a vector of multivariate polynomials (with variables $t, \alpha_1, \ldots, \alpha_{|S|}$), recursively, as follows:

\[ \mathbf{q}^{S'}_i(t; \alpha_1, \ldots, \alpha_{|S|}) = \left\{ \begin{array}{cc} \displaystyle \alpha_i &\text{ if } i\in S \\
\displaystyle t \sum_{j \in N[i] \setminus \{k\}} \mathbf{q}^{S'}_j(t;  \alpha_1, \ldots, \alpha_{|S|})   &\text{ if } i \text { forces } k \end{array} \right. \]
\end{mydef}

\begin{prop} For a graph $G$ with a zero forcing set $S$ and forcing chain $S'$, 
\[ \mathbf{q}^{S'} (t; 1, \ldots, 1) = \mathbf{q}^{S'} (t). \]
\end{prop}

The proof follows by setting $\alpha_i = 1$ for all $i$. \hfill \qed

\begin{mydef}[Variance polynomial vector of a forcing chain] 
Let $S$ be a zero forcing set of $G$ with forcing chain $S'$.
We define the {\it  variance polynomial vector} of $S'$, $\mathbf{V}^{S'}(t)$ as a vector of polynomials:

\[ \mathbf{V}^{S'}(t) = \sum_{i \in S} \left[  [{\alpha_i}] \mathbf{q}_i^{S'}(t ;\alpha_1, \ldots, \alpha_{|S|})  \right]^2 \]  
where $[\alpha_i](p(\cdots))$ is the coefficient of $\alpha_i$, which is a polynomial in $t$.

\end{mydef}

\begin{theorem}\label{laprop3}
Let $G$ be a graph. Fix  a zero forcing set $S$ with forcing chain $S'$ and variance polynomial vector $\mathbf{V}^{S'}(t)$. Then, \textbf{for any} $\mathbf{A}\in \mathcal{S}_\mathbb{R}(G)$  with multiplicative row-support spread $\kappa' := \kappa'(A)$, if $\mathbf{A} \mathbf{x} =  \mathbf{0}$ and $\mathbf{x}_i'$ is a random variable with mean ${\mathbf{x}}_i$ and variance $\varepsilon$, independent from all other $i\in S$, by only using $\mathbf{x}_S'$, the entries of $\mathbf{x}'$ corresponding to $S$, 
a vector $\hat{\mathbf{x}}$ can be computed so that the random quantity $\mathbf{x}_i -\hat{\mathbf{x}}_i$ has mean 0 and variance at most  $\mathbf{V}^{S'}_i( \kappa') ~ \varepsilon$.
\end{theorem}

\begin{proof} Without loss of generality, label the vertices so that $S = \{1, \ldots, |S|\}$. For $i \in S$, take $\hat{\mathbf{x}}_i = \mathbf{x'}_i$; for $k \not \in S$, as in the spirit of Theorem \ref{laprop}, take $\hat{\mathbf{x}}_k$ to be defined recursively based upon the vertex $v$ that forces $k$:

\begin{equation}
 \hat{\mathbf{x}}_{k} := \frac{-1}{\mathbf{A}_{vk}} \sum_{j \in N[v] \setminus \{k\}}\mathbf{A}_{vj} \hat{\mathbf{x}}_j, \notag
\end{equation}

and just as in the proof of Theorem \ref{laprop}, we have
\begin{eqnarray} {\mathbf{x}}_{k} - \hat{\mathbf{x}}_k &=& \frac{-1}{\mathbf{A}_{vk}}  \sum_{j \in N[v] \setminus \{k\}} \mathbf{A}_{vj} \mathbf{x}_j - \frac{-1}{\mathbf{A}_{ik}} \sum_{j \in N[v] \setminus \{k\}}\mathbf{A}_{vj} \hat{\mathbf{x}}_j \notag  \\
&=& \frac{-1}{\mathbf{A}_{vk}}  \sum_{j \in N[v] \setminus \{k\}}\mathbf{A}_{vj} (\mathbf{\hat x}_j - \mathbf{x}_j)  \label{comb1} .
 \end{eqnarray}

For $i \in S$, Let $\alpha_i$ denote the random quantity $\mathbf{x}_i -\hat{\mathbf{x}}_i$.

%Claim 1: For $i \not\in S$, the random quantity $\mathbf{x}_i -\hat{\mathbf{x}}_i$ can be expressed as a (non-random) linear combination of the random quantities $\alpha_1, \ldots,  \alpha_{|S|}$.
%\begin{proof}[Proof of Claim 1:]
%The proof follows from inductively applying line \ref{comb1} above. \renewcommand\qedsymbol{$\triangle$}
%\end{proof}

By inductively applying line \ref{comb1} we have that for any $k \not\in S$, the random quantity $\mathbf{x}_k -\hat{\mathbf{x}}_k$ can be expressed as a (non-random) linear combination of the random quantities $\alpha_1, \ldots,  \alpha_{|S|}$.  For the remainder, we will let $\hat{\mathbf{x}}_k - \mathbf{x}_k = \sum_{i \in S} C_{i k} \alpha_i$ be this  linear combination of $\hat{\mathbf{x}}_k - \mathbf{x}_k$ over the $\alpha_i$.

Since, by hypothesis, for each $i \in S$, the random quantity $\hat{\mathbf{x}}_i := \mathbf{x}'_i$ has mean $\mathbf{x}_i$, the random quantity $\alpha_i$ has mean 0. Since for any $k \not\in S$, $\hat{\mathbf{x}}_k- \mathbf{x}_k$  is a linear combination of $\alpha_1, \ldots,  \alpha_{|S|}$, it follows by linearity of expectation that $\hat{\mathbf{x}}_k- \mathbf{x}_k$ must also have mean 0 as well.

Therefore, it remains to show that for all $i$, the variance of $\hat{\mathbf{x}}_k - \mathbf{x}_k$ is bounded by $\mathbf{V}^{S'}_i( \kappa'; ) ~ \varepsilon$.

Claim 1: For any vertex $k$, \[ |C_{i k}| \le | [{\alpha_i}] \mathbf{q}_k^{S'}(\kappa' ;\alpha_1, \ldots, \alpha_{|S|}) |\]
where the right side the coefficient of ${\alpha_i}$ in $\mathbf{q}_k^{S'}(\kappa' ;\alpha_1, \ldots, \alpha_{|S|}).$

\begin{proof}
We proceed on induction on $m, 0 \le m \le \tau = pt(G,S)$, the number of iterations of the color change rule needed to color $k$. (If $k \in S$, $m=0$).

Base case ($m=0$). For $m=0$, necessarily $k \in S$. Therefore, by the definition of $\mathbf{q}_i^{S'}(\kappa' ;\alpha_1, \ldots, \alpha_{|S|})$, $[{\alpha_i}] \mathbf{q}_i^{S'}(\kappa' ;\alpha_1, \ldots, \alpha_{|S|}) = 1$ if $i = k$ and $0$ otherwise. Hence, $C_{i k} = \mathbf{q}_i^{S'}(\kappa' ;\alpha_1, \ldots, \alpha_{|S|})$, and the claim holds.

Induction step ($m=T$). Suppose, the claim holds for $m=T-1$. Then, for $k$ forced by $v$ at iteration $T$, we have

\begin{eqnarray*}
|C_{ik}|  &=&  \left|  [\alpha_i]  \left( \sum_{j \in N[v] \setminus \{k\}} \frac{\mathbf{A}_{vj}}{\mathbf{A}_{vk}} (\mathbf{x}_j-\hat{\mathbf{x}}_j) \right) \right| \\
&=&  \left|  \sum_{j \in N[v] \setminus \{k\}} \frac{\mathbf{A}_{vj}}{\mathbf{A}_{vk}} [\alpha_i] \left(\mathbf{x}_j-\hat{\mathbf{x}}_j \right) \right| \\
%&\le&  \left|    \sum_{j \in N[v] \setminus \{k\}}\frac{\mathbf{A}_{vj}}{\mathbf{A}_{vk}} C_{ij} \right| \\
&\le&  \left|   \sum_{j \in N[v] \setminus \{k\}} \kappa' C_{ij}  \right| \\
&\le&  \left|   \sum_{j \in N[v] \setminus \{k\}} \kappa' \cdot [{\alpha_i}] \mathbf{q}_j^{S'}(\kappa' ;\alpha_1, \ldots, \alpha_{|S|})  \right| \\
&=&  \left|  [\alpha_i] \left(  \sum_{j \in N[v] \setminus \{k\}} \kappa' \cdot \mathbf{q}_j^{S'}(\kappa' ;\alpha_1, \ldots, \alpha_{|S|}) \right) \right| \\
&=&  \left|    [{\alpha_i}]  \mathbf{q}_k^{S'}(\kappa' ;\alpha_1, \ldots, \alpha_{|S|})  \right| \\
\end{eqnarray*}
where the fourth-to-last line follows from the definition of $\kappa'$, third-to-last line follows from the induction hypothesis, and the last line follows from the definition of $ \mathbf{q}_k^{S'}(\kappa' ;\alpha_1, \ldots, \alpha_{|S|}) $.

\renewcommand\qedsymbol{$\triangle$}
\end{proof}

Recall that for random independent variables $X_1, \ldots X_n$, \[ {{\rm Var}} \left( \sum_{i=1}^n c_i X_i\right) =  \sum_{i=1}^n c_i^2 {\rm Var}(X_i).\] 

Therefore,
\begin{eqnarray*}
{\rm Var}( \hat{\mathbf{x}}_k- \mathbf{x}_k) &=&  {\rm Var} \left(\sum_{i \in S} C_{i k} \alpha_i  \right)\\
&=&  \sum_{i \in S} C_{i k}^2 {\rm Var}(\alpha_i ) \\
&\le&  \sum_{i \in S} \left[  [{\alpha_i}] \mathbf{q}_i^{S'}(t ;\alpha_1, \ldots, \alpha_{|S|})  \right]^2 \varepsilon  \\
&=& \mathbf{V}_i^{S'}(t) \varepsilon
\end{eqnarray*}
where the last line follows from Claim 1 and that $ {\rm Var}( \hat{\mathbf{x}}_i - \mathbf{x}_i) = {\rm Var}( \mathbf{x}'_i) = {\rm Var}( \hat{\mathbf{x}}_i) =  \varepsilon$ as given in the hypothesis of the theorem.

\end{proof}

As with $\mathbf{q}^{S'}(t)$, $\mathbf{V}^{S'}(t)$ is defined for a given forcing chain. In actuality, one can apply all possible forcing chains simultaneously and choose the ``best'' entry. As a result, for a zero forcing set $S$, we can define $\mathbf{V}^S(t)$ similiar to $\mathbf{q}^S(t)$:

\begin{mydef}[Variance polynomial vector for a zero forcing set]
Given a graph $G$ and a zero forcing set $S$, define the variance polynomial vector of $S$
\[ \mathbf{V}^{S}_i = \min_{S',\preceq} \mathbf{V}^{S'}_i \]
where the minimum is taken over all forcing chains $S'$ using the zero forcing set $S$, under the order $\preceq$.
\end{mydef}

We emphasize again that $\mathbf{V}^{S}_i$ is defined entry-wise, and {\it in contrast to} $\mathbf{q}^{S}$, there may not be a single forcing chain, $S'$ that achieves $\mathbf{V}^{S}_i =  \mathbf{V}^{S'}_i$ for all $i$ (See Section \ref{vsnevsprime}).

\section{Examples and Constructions}
\label{sec:exs}

\subsection{An example where different forcing sets achieve optional $\mathbf{q}^S(t)$ and $\mathbf{V}^S(t)$.}

To see how the error polynomial vector and the variance polynomial vector can be used to compare different zero forces sets, in Figure \ref{differzfs} is a graph, $G$, on 9 vertices, ``\{9, 3094\}'' in the {\it Wolfram Database}. This graph has $Z(G) = 3$ and a minimal propagation time of 4. However, Figure \ref{differzfs} illustrates different minimum zero forcing sets with different $\mathbf{q}^S(t)$ and $\mathbf{V}^S(t)$. In the context of error of a particular system, it would be best to consider maximum error or maximum variance over all the vertices in the graph, in which case, we can measure a zero forcing set based upon the maximum entry of the corresponding polynomial vector based on the ordering $\preceq$. The peculiar aspect about $G$ is that the zero forcing sets that achieve the minimum maximum entry {\it are different} under $\mathbf{q}^S(t)$ than for $\mathbf{V}^S(t)$. For instance, in Figure \ref{differzfs}, the first column is the set that achieves the minimum for the maximum entry of $\mathbf{V}^S(t)$ whereas the second column is the set that achieves the minimum for the maximum entry of $\mathbf{q}^S(t)$. In short, which zero forcing set is ``best'' depends upon whether one expects the error to be independent/uncorrelated (in which case use $\mathbf{V}^S(t)$) or absolute/highly-correlated (in which case use $\mathbf{q}^S(t)$). Further, certain sets with the same propagation time may have significantly worse maximum error as seen in the last column.

\begin{figure}[h]
\centering
$\begin{array}{| c | c | c| c |}
\hline
& %picture1
\begin{tikzpicture}
\node[draw=none] (topstrut) at (0,0.2) {};
\node (1) at (0,0) {};
\node[fill=blue] (2) at (1,0) {};
\node (3) at (2,0) {};
\node (4) at (0,-1) {};
\node (5) at (1,-1) {};
\node[fill=blue] (6) at (2,-1) {};
\node (7) at (0,-2) {};
\node (8) at (1,-2) {};
\node[fill=blue] (9) at (2,-2) {};
\draw (1) -- (4);
\draw (2) -- (5) -- (8);
\draw (3) -- (6) -- (9);
\draw (4) -- (5) -- (6);
\draw (7) -- (8) -- (9);
\draw (5) -- (9);
\end{tikzpicture}
& %picture2
\begin{tikzpicture}
\node[draw=none] (topstrut) at (0,0.2) {};
\node[fill=blue] (1) at (0,0) {};
\node (2) at (1,0) {};
\node (3) at (2,0) {};
\node (4) at (0,-1) {};
\node (5) at (1,-1) {};
\node[fill=blue] (6) at (2,-1) {};
\node (7) at (0,-2) {};
\node (8) at (1,-2) {};
\node[fill=blue] (9) at (2,-2) {};
\draw (1) -- (4);
\draw (2) -- (5) -- (8);
\draw (3) -- (6) -- (9);
\draw (4) -- (5) -- (6);
\draw (7) -- (8) -- (9);
\draw (5) -- (9);
\end{tikzpicture}
& %picture3
\begin{tikzpicture}
\node[draw=none] (topstrut) at (0,0.2) {};
\node[fill=blue] (1) at (0,0) {};
\node (2) at (1,0) {};
\node (3) at (2,0) {};
\node[fill=blue] (4) at (0,-1) {};
\node (5) at (1,-1) {};
\node (6) at (2,-1) {};
\node[fill=blue] (7) at (0,-2) {};
\node (8) at (1,-2) {};
\node (9) at (2,-2) {};
\draw (1) -- (4);
\draw (2) -- (5) -- (8);
\draw (3) -- (6) -- (9);
\draw (4) -- (5) -- (6);
\draw (7) -- (8) -- (9);
\draw (5) -- (9);
\end{tikzpicture}
\\ \hline
\max_i \mathbf{q}^S_i (t) & t^4+3 t^3+4 t^2  & t^4 + 2 t^3 +  4 t^2 + 2 t   & 3 t^4+7 t^3+4 t^2+t  \\\hline
\max_i \mathbf{V}^S_i (t) & \thead{ t^8+2 t^7+7 t^6+8 t^5+ 6 t^4} & \thead{   t^8 + 4 t^7 + 8 t^6 + 8 t^5 \\ \quad  6 t^4 + 4 t^3  + 2 t^2} & \thead{ 3 t^8+14 t^7+25 t^6+22 t^5 \\ +10 t^4+2 t^3+t^2}  \\\hline
\end{array}$
\caption{Different zero forcing sets and their maximum entries for $\mathbf{q}^S(t)$ and $\mathbf{V}^S(t)$ under the ordering $\preceq$}
\label{differzfs}
\end{figure} 

\subsection{Family of graphs where different forcing sets achieve optimal $\mathbf{q}^S(t)$ and $\mathbf{V}^S(t)$.}

For $n\geq 7$, let $G$ be the graph obtained by adding a duplicating a leaf of a path on $n-1$ vertices, as illustrated in Figure~\ref{fig:forkedpath}.  Up to symmetry, there are only three minimum zero forcing sets, namely, $S_1=\{1,n\}$, $S_2=\{n-3,n\}$, and $S_3=\{n-1,n\}$.  Both $S_1$ and $S_2$ has the propagation time $n-3$, yet $S_3$ has the propagation time $n-2$, so $S_3$ cannot be an optimal zero forcing set.  For $n=7$, we may compute 
\[\mathbf{q}^{S_1}(t)=
\begin{bmatrix}
   &   &   &   & 1 \\
   &   &   & t &   \\
   &   & t^2 & +t &   \\
   & t^3 & +2t^2 &   &   \\
   &   &   & t &   \\
 t^4 & +2t^3 & +t^2 & +t &   \\
   &   &   &   & 1
\end{bmatrix}\text{ and }
\mathbf{q}^{S_2}(t)=\begin{bmatrix}
  &   & t^4 & +2t^3 & +2t^2 &   &   \\
  &   &   & t^3 & +t^2 & +t &   \\
  &   &   &   & t^2 & +t  &   \\
  &   &   &   &   &   & 1 \\
  &   &   &   &   & t &   \\
  &   &   &   & t^2 & +2t &   \\
  &   &   &   &   &   & 1
\end{bmatrix}.\]
Thus, $S_1$ is optimal in terms of the error polynomial vector.  On the other side,
\[\mathbf{V}^{S_1}(t)=\begin{bmatrix}
  &   &   &   &   &   &   &   & 1 \\
  &   &   &   &   &   & t &   &   \\
  &   &   &   & t^3 & +2t^2 & +t &   &   \\
  &   & t^5 & +4t^4 & +4t^3 &   &   &   &   \\
  &   &   &   &   &   & t &   &   \\
t^7 & +4t^6 & +4t^5 &   & +t^3 & +2t^2 & +t &   &   \\
  &   &   &   &   &   &   &   & 1
\end{bmatrix}\text{ and }\]
\[\mathbf{V}^{S_2}(t)=\begin{bmatrix}
t^7 & +2t^6 & +2t^5 & +4t^4 & +4t^3 &   &   &   &   \\
  &   & t^5 &   & +t^3 & +2t^2 & +t &   &   \\
  &   &   &   & t^4 &   & +t^2 &   &   \\
  &   &   &   &   &   &   &   & 1 \\
  &   &   &   &   &   & t^2 &   &   \\
  &   &   &   & t^4 & +2t^3 & +2t^2 &   &   \\
  &   &   &   &   &   &   &   & 1
\end{bmatrix}.\]
Therefore, $S_2$ is optimal in terms of the variance polynomial vector.  This behavior also happens when $n=8$; inductively, it happens for all $n\geq 7$.  

Here are some intuitive explanations.  In terms of the error polynomials, $S_2$ requires the errors of $n$ and $n-3$ being carried for a long way to vertex $1$, so it is not a good choice comparing to $S_1$.  In terms of the variance polynomials, the error of vertex $1$ for $S_2$ are evenly contributed by $\alpha_{n-3}$ and $\alpha_{n}$, causing a smaller variance.  Thus, it is a better choice than $S_1$. 

\begin{figure}[h]
\begin{center}
\begin{tikzpicture}
\node[label={$1$}] (1) at (0,0) {};
\node[label={$2$}] (2) at (1,0) {};
\node[label={$3$}] (3) at (2,0) {};
\node[label={$n-3$}] (33) at (4,0) {};
\node[label={right:$n-2$}] (22) at (5,0) {};
\node[label={right:$n-1$}] (11) at (6,1) {};
\node[label={right:$n$}] (00) at (6,-1) {};
\draw (1) -- (2) --(3);
\draw (22) -- (33);
\draw (11) -- (22) -- (00);
\draw[dotted] ([xshift=0.5cm]3.east) -- ([xshift=-0.5cm]33.west);
\end{tikzpicture}
\end{center}
\caption{A forked path}
\label{fig:forkedpath}
\end{figure}
  
\subsection{Example where no single forcing chain yields all of the polynomial entries for $\mathbf{V}^S(t)$ } \label{vsnevsprime}

%Given a graph $G$ and a zero forcing set $S$, we may consider the variance polynomial vector $\mathbf{V}^S(t)$ of this zero forcing set (instead of a given zero forcing chain) as the entrywise minimum of vector $\mathbf{V}^{S'}(t)$ over all possible forcing chain $S'$ of $S$.  
Unlike Propsition~\ref{sameq}, we will show that for the graph and the zero forcing set in Figure~\ref{fig:notsamevar}, its variance polynomial vector of the zero forcing set is not achieved by the variance polynomial vector of any forcing chain.

Let $G$ be the graph shown in Figure~\ref{fig:notsamevar} and $S=\{1,3,5\}$ a minimum zero forcing set of $G$.  Let $\mathbf{V} :=\mathbf{V}^S(t)$ be the variance polynomial vector of $S$.  We will show that $\mathbf{V} (t)\neq \mathbf{V}^{S'}(t)$ for any forcing chain $S'$ of $S$.  Suppose, for the purpose of yielding a contradiction, there is a forcing chain $S'$ with $\mathbf{V}^{S'}(t)=\mathbf{V}(t)$.  Let $\mathbf{q}=\mathbf{q}^{S'}(t;\alpha_1,\alpha_3,\alpha_5)$ be the multivariate error polynomial vector of $S'$.  Since $S=\{1,3,5\}$, we know $\mathbf{q}_1=\alpha_1$, $\mathbf{q}_3=\alpha_3$, and $\mathbf{q}_5=\alpha_5$.  If $5\rightarrow 8$, then $\mathbf{q}_8$ has degree $1$.  If $8$ is not forced by $5$, $\mathbf{q}_8$ will have its degree too high, causing the degree of $\mathbf{V}_8^{S'}$ too high.  Thus, we know $5\rightarrow 8$ and $\mathbf{q}_8=t\alpha_5$.  Similarly, it must be $3\rightarrow 6$ and $\mathbf{q}_6=t\alpha_3$.  For a similar reason, $6\rightarrow 2$ make $\mathbf{q}_2$ have degree $2$ and is optimal.  Thus $\mathbf{q}_2={t\alpha_1}+(t^2+t)\alpha_3$.  Now comes the first and the only fork.  It can be $8\rightarrow 7$ or $1\rightarrow 7$ making $\mathbf{q}_7$ degree $2$.  (Here, $2\rightarrow 7$ is impossible since it gives $\mathbf{q}_7$ degree $3$.)

\textbf{Case 1:} $8\rightarrow 7$ and $7\rightarrow 4$.  In this case we have 
\[\left\{\begin{aligned}\mathbf{q}_7 &= (t^2+t) \alpha_5\\
 \mathbf{q}_4 &= (t^2+t)\alpha_1 + (t^3+t^2)\alpha_3 + (t^3+2t^2)\alpha_5
\end{aligned}\right.,\text{ so }
\left\{
\begin{aligned}
\mathbf{V}_7^{S'} &= t^4 + 2t^3 + t^2 \\
\mathbf{V}_4^{S'} &= 2 t^6 + 6 t^5 + 6 t^4 + 2 t^3 + t^2
\end{aligned}\right..\]

\textbf{Case 2:} $1\rightarrow 7$ and $7\rightarrow 4$.  In this case we have 
\[\left\{\begin{aligned}\mathbf{q}_7 &=  t \alpha_1 + t^2 \alpha_3\\
 \mathbf{q}_4 &= (2t^2+t) \alpha_1 + (2t^3+t^2) \alpha_3 + t^2 \alpha_5
\end{aligned}\right.,\text{ so }
\left\{
\begin{aligned}
\mathbf{V}_7^{S'} &= t^4 + t^2 \\
\mathbf{V}_4^{S'} &= 4 t^6 + 4 t^5 + 6 t^4 + 4 t^3 + t^2 
\end{aligned}\right..\]
 
Thus, Case 1 is better for $\mathbf{V}_4^{S'}$ but Case 2 is better for $\mathbf{V}_7^{S'}$.  Indeed, through this argument we also know 

\[\mathbf{V}=
\begin{bmatrix}
  &  &  &  &  &  & 1 \\
  &  & t^4 &+2t^3 & +2t^2 &  &  \\
  &  &  &  &  &  & 1 \\
 2t^6 &+6t^5 & +6t^4 &+2t^3 &+t^2 &  &  \\
  &  &  &  &  &  & 1 \\
  &  &  &  & t^2 &  &  \\
  &  & t^4 &  &+t^2 &  & \\
  &  &  &  & t^2 &  &  \\
\end{bmatrix}.\]

%A side note:  This is a minimal example in terms of the number of vertices.  There is no example with 7 or fewer vertices.  

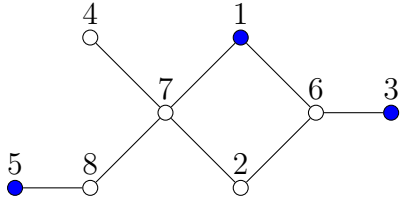
\begin{figure}[h]
\begin{center}
\begin{tikzpicture}
\node [label={$8$}] (8) at (0,0) {};
\node [fill=blue,label={$1$}] (1) at (2,2) {};
\node [label={$2$}] (2) at (2,0) {};
\node [fill=blue,label={$3$}] (3) at (4,1) {};
\node [label={$4$}] (4) at (0,2) {};
\node [fill=blue,label={$5$}] (5) at (-1,0) {};
\node [label={$6$}] (6) at (3,1) {};
\node [label={$7$}] (7) at (1,1) {};
\draw (5) -- (8) -- (7) -- (2) -- (6) -- (1) --(7) -- (4);
\draw (6) -- (3);
\end{tikzpicture}
\end{center}
\caption{A graph and its zero forcing set $S$ such that $\mathbf{V}^S\neq \mathbf{V}^{S'}$ for all forcing chain $S'$ of $S$}
\label{fig:notsamevar}
\end{figure}

\subsection{Explicit calculation of $\mathbf{q}^{\{1\}}(t)$ for $P_n$}

Consider a path $P_n$ on $n$ vertices with the vertices labeled as $1,\ldots ,n$ in order.  Then $S=\{1\}$ is a minimum zero forcing set, and we know $\mathbf{q}_1^S(t)=1$, $\mathbf{q}_2^S(t)=t$, and $\mathbf{q}_k^S(t)=t(\mathbf{q}_{k-1}^S(t)+\mathbf{q}_{k-2}^S(t))$ for $k=3,\ldots, n$.  For example, the error polynomial vector of $P_{15}$ with $S=\{1\}$ is \[\begin{bmatrix}
 &  &  &  &  &  &  &  &  & 1 \\
 &  &  &  &  &  &  &  & t &  \\
 &  &  &  &  &  &  & t^2 & +t &  \\
 &  &  &  &  &  & t^3 & +2t^2 &  &  \\
 &  &  &  &  & t^4 & +3t^3 & +t^2 &  &  \\
 &  &  &  & t^5 & +4t^4 & +3t^3 &  &  &  \\
 &  &  & t^6 & +5t^4 & +6t^3 & +t^2 &  &  &  \\
 &  & t^7 & +6t^6 & +10t^5 & +4t^4 &  &  &  &  \\
 & t^8 & +7t^7 & +15t^6 & +10t^5 & +t^4 &  &  &  &  \\
t^9 & +8t^8 & +21t^7 & +20t^6 & +5t^5 &  &  &  &  & 
\end{bmatrix}.\]
Indeed, if let $T(a,b)$ be the number of subsets of $\{1,...,a\}$ of size $k$ that contain no consecutive integers, then 
\[[t^{k-1-r}]\mathbf{q}_k^S(t)=T(k-2,r).\]
This is because $T(a,b)=\binom{a-b+1}{b}$ and $T(a,b)=T(a-1,b)+T(a-2,b-1)$, which meets the recurrence relation of $\mathbf{q}_k^S$.  For instance, the $21t^7$ comes from $k=10$, $r=2$, and $T(8,2)=\binom{8-2+1}{2}=21$.  This integer sequence is labeled as \textsf{A011973} in OEIS and also related to the Fibonacci polynomials; and more information can be found on \cite{OEIS1}.

\subsection{A graph where $\mathbf{q}^S(t)$ and $\mathbf{V}^S(t)$ have different leading coefficients for optimal graphs}

Let $G$ be the graph in Figure~\ref{distleadcoef} and $S=\{4,5\}$ its minimum zero forcing set.  Then the 
\[\mathbf{q}^S(t)=
\begin{bmatrix}
 4t^2 & +t &  \\
  & 2t &  \\
  & 2t &  \\
  &  & 1 \\
  &  & 1
\end{bmatrix}\text{ and }\mathbf{V}^S(t)=
\begin{bmatrix}
 8t^4 & +4t^4 & +t^3 &  &  \\
  &  & 2t^3 &  &  \\
  &  & 2t^3 &  &  \\
  &  &  &  & 1 \\
  &  &  &  & 1
\end{bmatrix}.
\]
This shows that the leading coefficients of the maximum entries might not be the same.

\begin{figure}[h]
\begin{center}
\begin{tikzpicture}
\node[label={$1$}] (1) at (0,0) {};
\node[label={$2$}] (2) at (1,0) {};
\node[label={$3$}] (3) at (2,1) {};
\node[fill=blue,label={$4$}] (4) at (2,-1) {};
\node[fill=blue,label={$5$}] (5) at (3,0) {};
\draw (1) -- (2) -- (3) -- (5) -- (4) -- (2);
\end{tikzpicture}
\end{center}
\caption{A graph and a zero forcing set with distinct leading coefficients for maximum entries of the error polynomial vector and the variance polynomial vector}
\label{distleadcoef}
\end{figure}
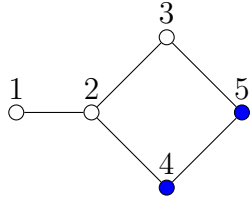

\section{Discussion}

In the introduction, we mentioned that previous work focused on the optimizing $|S| + pt(G,S)$. In particular, it is sometimes possible that the zero forcing set optimizing $|S| + pt(G,S)$ is not a minimum zero forcing set. This phenomenon is known as ``throttling.'' With the new tools of the error polynomial vector and the variance polynomial vector, it would be interesting to revisit throttling. Namely, we ask: Is it possible to substantially decrease the error polynomial vector and/or the variance polynomial vector by adding a few vertices to the zero forcing set, especially in cases were the propagation time may remain unchanged?

The objective $|S| + pt(G,S)$ seems arbitrary. However, with the error polynomial vector and the variance polynomial vector, one may be able to develop a more meaningful objective regarding the trade off between the size of the zero forcing set (i.e., the number of sensors) and the potential error. For instance, if each sensor has a cost and the error of each estimate has a penalty, then one would have a different objective function that would almost certainly be different from $|S| + pt(G,S)$. What are these objective functions, and in what ways do zero forcing sets affect them?

The parameter $\kappa'(\mathbf{A})$ may be interesting to investigate in its own right, especially in the context of the minimum rank problem. In particular, for certain graphs $G$, the matrix attaining the maximum nullity (or minimum rank) might necessarily have $\kappa'(\mathbf{A}) > 1$; therefore, the minimum rank may increase whenever  $\kappa'(\mathbf{A})$ is constrained from above.

%\section{Acknowledgements}
%
%For the first author, this work is partially supported by NARC (Naval Academy Research) Grant as well by NSF Grant DMS-1719894.

%\section{Proof of Proposition}
%\begin{proof}
%\end{proof}

%\bibliographystyle{acm}
%\bibliography{bibMinNumRank}

\end{document}